\documentclass[12pt]{article}%
\usepackage{amsmath, amsfonts, amsthm, color,latexsym}
\usepackage{amsmath}
\usepackage{amsfonts}
\usepackage[all] {xy}
\usepackage{amssymb}
\usepackage{color} 
\allowdisplaybreaks[4]
\usepackage{graphicx}%
\providecommand{\U}[1]{\protect\rule{.1in}{.1in}}
\newtheorem{teo}{Theorem}[section]
\newtheorem{prop}[teo]{Proposition}
\newtheorem{cor}[teo]{Corollary}
\newtheorem{ex}[teo]{Example}
\newtheorem{obs}[teo]{Remark}

\newtheorem{lema}[teo]{Lemma}
\newtheorem{final remark}[teo]{Final Remark}
\newtheorem{definition}[teo]{Definition}
\newcommand{\an}{\left \Vert} 

\newcommand{\fn}{\right \Vert} 

\newcommand{\ach}{\left \{} 

\newcommand{\fch}{\right \}} 

\newcommand{\ap}{\left (} 

\newcommand{\fp}{\right )} 

\begin{document}

\title{\sc Hyper-ideals of multilinear operators}
\date{}
\author{Geraldo Botelho\thanks{Supported by CNPq Grant
305958/2014-3 and Fapemig Grant PPM-00326-13.\hfill\newline2010 Mathematics Subject
Classification: 47L22, 46G25, 47L20, 47B10, 46B28, 46A32. \newline Keywords: Banach spaces, multilinear operators, operator ideals, multi-ideals, hyper-ideals.}  ~and Ewerton R. Torres}\maketitle

\begin{abstract} We introduce and develop the notion of hyper-ideals of multilinear operators between Banach spaces. While the well studied notion of ideals of multilinear operators (multi-ideals) relies on the composition with linear operators, the notion we propose, by considering the composition with multilinear operators, explores more deeply the nonlinear feature of the subject. The results we prove show that, although more restrictive {\it a priori}, hyper-ideals enjoy nice general properties and its theory is rich enough to provide distinguished examples in every situation where the corresponding multi-ideal fails to be a hyper-ideal.
\end{abstract}

\section{Introduction and background}
The theory of ideals of multilinear operators between Banach spaces (multi-ideals) was initiated by A. Pietsch \cite{pietsch} as a first step to take to the nonlinear setting the successful theory of ideals of linear operators (operator ideals). Since then much research has been done in this subject, we mention just a few recent developments: \cite{achour, aronrueda3, aronrueda2, erhan1, erhan2, pablo, jamilson, carando2, carandoracsam, defant, joilson, popan, popaarchiv, popass, diana}.

In this paper we introduce a refinement of the concept of multi-ideals, which we call {\it hyper-ideals}, with the purpose of exploring the stability of the class with respect to the composition with multilinear operators, rather than with respect to the composition with linear operators as in the case of multi-ideals. To be more precise, let us recall the defining property of multi-ideals: a class $\cal M$ of multilinear operators enjoys the multi-ideal property if, in the following diagram,
\begin{displaymath}
\xymatrix@C=-3pt{
G_1 \ar[d]^{u_1}& \times & G_2 \ar[d]^{u_2}& \times &\cdots &\times& G_n\ar[d]^{u_n}\ar[drrrrrrrrrrrrrrrrrrrrrrrrrrrrrrrrrrrr]^{~~t\circ A\circ(u_1,u_2,\ldots,u_n)} &&&&&&&&&&&&&&&&&& & \\
E_1 & \times & E_2 & \times & \cdots &\times& E_n \ar[rrrrrrrrrrrrrrrrrr]^{A}&&&&&&&&&&&&&&&&&&  F \ar[rrrrrrrrrrrrrrrrrr]^{t~~~} &&&&&&&&&&&&&&&&&& H}
\end{displaymath}
$u_1, u_2, \ldots, u_n,t$ are continuous linear operators and $A$ is an $n$-linear operator belonging to $\cal M$, then the composition  $t\circ A\circ(u_1,\ldots,u_n)$ belongs to $\cal M$ as well. The following question is quite natural: once we are in the multilinear setting, why not considering the composition of $A$ with {\it multilinear} operators on the left-hand side? This leads us to the consideration of classes $\cal H$ of multilinear operators enjoying the {\it hyper-ideal property}: if, in the following diagram,
{\footnotesize\begin{displaymath}
\xymatrix@C=-3pt{
(G_1\times\cdots\times G_{m_1}) \ar[d]^{B_1}& \times & (G_{m_1+1}\times\cdots\times G_{m_2}) \ar[d]^{B_2}& \times &\cdots &\times& (G_{m_{n-1}+1}\times\cdots\times G_{m_n})\ar[d]^{B_n}\ar[drrrrrrrrrrrrrrrrrrrrrrrrrrrrrrrrrrr]^{~~~~~~~t\circ A\circ(B_1,\ldots,B_n)\in\mathcal{H}} &&&&&&&&&&&&&&&&&& & \\
E_1 & \times & E_2 & \times & \cdots &\times& E_n \ar[rrrrrrrrrrrrrrrrrr]^{A\in\mathcal{H}}&&&&&&&&&&&&&&&&&&  F \ar[rrrrrrrrrrrrrrrrr]^{t~~~~~~~}& &&&&&&&&&&&&&&&& H}
\end{displaymath}}
$B_1, \ldots, B_n$ are {\it multilinear} operators, $t$ is a linear operator and $A$ is an $n$-linear operator belonging to $\cal H$, then the composition  $t\circ A\circ(B_1,\ldots,B_n)$ belongs to $\cal H$ as well.

The hyper-ideal property has already been studied individually for some classes of multilinear operators, see, e.g., \cite{defant, popan, popass}. The purpose of this paper is to systematize the study of the classes satisfying this property, which we call {\it hyper-ideals}. In Section 1 we develop the basics of the theory of hyper-ideals. Sections 2 and 3 have several purposes: (i) Many illustrative examples of hyper-ideals are provided, including classical classes such as compact and weakly compact multilinear operators. (ii) Of course every hyper-ideal is a multi-ideal; some important multi-ideals are shown not to be hyper-ideals. (iii) Deep distinctions with the theory of multi-ideals are established, for instance Corollary \ref{difmulti}. (iv) Whenever an important multi-ideal fails to be a hyper-ideal, we construct/identify a hyper-ideal that enjoys the properties, in the hyper-ideal context, of the missing multi-ideal. For example, as soon as we establish that the class of nuclear multilinear operators, which is smallest Banach multi-ideal, fails to be a hyper-ideal, we construct the class of hyper-nuclear operators, which shall be proved to be the smallest Banach hyper-ideal. This and other analogous situations show that the theory of hyper-ideals is rich and independent from the theory of multi-ideals.

From now on, $E, F, G, H, E_n, G_n, n \in \mathbb{N}$, shall denote Banach spaces over $\mathbb{K} = \mathbb{R}$ or $\mathbb{C}$. The symbols $E'$ stands for the topological dual of $E$ and $B_E$ for its closed unit ball. By ${\cal L}(E_1, \ldots, E_n;F)$ we denote the Banach space of continuous $n$-linear operators from $E_1 \times \cdots \times E_n$ to $F$ endowed with the usual uniform norm $\|\cdot\|$. In the linear case we write ${\cal L}(E;F)$. If $E_1  = \cdots = E_n$ we write ${\cal L}(^nE;F)$. If $F = \mathbb{K}$ we write ${\cal L}(E_1, \ldots, E_n)$ and ${\cal L}(^nE)$. Given $\varphi_1 \in E_1', \ldots, \varphi_n \in E_n'$ and $y \in F$, by $\varphi_1 \otimes \cdots \otimes \varphi_n \otimes y$ we mean the $n$-linear operator defined by
$$\varphi_1 \otimes \cdots \otimes \varphi_n \otimes y(x_1, \ldots, x_n) = \varphi_1(x_1) \cdots \varphi_n(x_n)b. $$
Linear combinations of such operators are called {\it $n$-linear operators of finite type}. A linear space-valued map is said to be of {\it finite rank} if its range generates a finite dimensional subspace of the target space. For the general theory of multilinear operators, see \cite{dineen, mujica}.

Given $0 < p \leq 1$, by $\ell_p(E)$ we denote the $p$-Banach (Banach if $p=1$) space of absolutely $p$-summable $E$-valued sequences endowed with its usual norm $\|\cdot\|_p$, and by $\ell_p^w(E)$ the $p$-Banach (Banach if $p=1$) space of weakly $p$-summable $E$-valued sequences with its usual norm $\|\cdot\|_{w,p}$ (see, e.g, \cite{diestel}).

Given a class $\cal H$ of multilinear operators between Banach spaces, by ${\cal H}^1$ we mean its linear component, that is, for all Banach spaces $E$ and $F$, ${\cal H}^1(E;F) := {\cal L}(E;F) \cap {\cal H}. $

A $p$-normed multi-ideal is a class $\cal M$ of multilinear operators endowed with a map $\|\cdot\|_{\mathcal{M}} \colon \mathcal{M} \longrightarrow [0,\infty)$ such that:\\
$\bullet$ For all $n, E_1, \ldots, E_n,F$, $(\mathcal{M}(E_1,\ldots, E_n;F), \|\cdot\|_{\mathcal{M}})$ is a $p$-normed linear subspace of $\mathcal{L}(E_1,\ldots, E_n;F)$ containing the $n$-linear operators of finite type;\\
$\bullet$  $\|I_n \colon \mathbb{K}^n\longrightarrow \mathbb{K}, I_n(\lambda_1,\ldots,\lambda_n)=\lambda_1\cdots\lambda_n\|_{\cal M} =1$ for every $n$;\\
$\bullet$ The multi-ideal property: If $A \in \mathcal{M}(E_1,\ldots, E_n;F)$, $u_1\in \mathcal{L}(G_{1};E_1)$, $\ldots$, $u_n\in \mathcal{L}(G_n;E_n)$ and $t \in \mathcal{L}(F;H)$, then  $t\circ A\circ(u_1,\ldots,u_n) \in \mathcal{M}(G_1,\ldots, G_n;H)$ and
$$\|t\circ A\circ(u_1,\ldots,u_n)\|_{\mathcal{H}}\le\|t\|\cdot\|A\|_{\mathcal{H}}\cdot
\|u_1\|\cdots\|u_n\|$$
(see the first diagram in this Introduction). The notions of normed, Banach and $p$-Banach multi-ideals are defined in the obvious way.

\section{Definition and basic properties}\label{idtop}

According to the philosophy described in the Introduction, we start the study of classes of multilinear operators that are stable with respect to the composition with multilinear operators on the left-hand side:

\begin{definition}\label{dhi}\rm A \textit{hyper-ideal of multilinear operators}, or simply a \textit{hyper-ideal}, is a subclass $\mathcal{H}$ of the class of all continuous multilinear operators between Banach spaces such that for all $n\in \mathbb{N}$ and Banach spaces $E_1, \ldots, E_n$ and $F$, the components $$\mathcal{H}(E_1,\ldots, E_n;F):=\mathcal{L}(E_1,\ldots, E_n;F)\cap \mathcal{H}$$ satisfy:\\
$(1)$ $\mathcal{H}(E_1,\ldots, E_n;F)$ is a linear subspace of $\mathcal{L}(E_1,\ldots, E_n;F)$ which contains the $n$-linear operators of finite type;\\
$(2)$ The hyper-ideal property: Given natural numbers $n$ and $1\le m_1<\cdots<m_n$, and  Banach spaces $G_1,\ldots,G_{m_n}$, $E_1,\ldots,E_n$, $F$ and $H$, if  $B_1\in \mathcal{L}(G_1,\ldots, G_{m_1};E_1), \ldots, B_n\in \mathcal{L}(G_{m_{n-1}+1},\ldots, G_{m_n};E_n)$, $t \in \mathcal{L}(F;H)$ and
$A \in \mathcal{H}(E_1,\ldots, E_n;F)$, then $t\circ A\circ(B_1,\ldots,B_n)$ belongs to $\mathcal{H}(G_1,\ldots, G_{m_n};H)$ (see the second diagram in the Introduction).

If there exist $p\in (0,1]$ and a map $\|\cdot\|_{\mathcal{H}} \colon \mathcal{H} \longrightarrow [0,\infty)$ such that:\\
(a) $\|\cdot\|_{\mathcal{H}}$ restricted to any component $\mathcal{H}(E_1,\ldots, E_n;F)$ is a $p$-norm;\\
(b) $\|I_n \colon \mathbb{K}^n\longrightarrow \mathbb{K}, I_n(\lambda_1,\ldots,\lambda_n)=\lambda_1\cdots\lambda_n\|_{\cal H} =1$ for every $n$;\\
(c) The hyper-ideal inequality: If $A \in \mathcal{H}(E_1,\ldots, E_n;F)$, $B_1\in \mathcal{L}(G_{1},\ldots, G_{m_1};E_1)$, $\ldots$, $B_n\in \mathcal{L}(G_{m_{n-1}+1},\ldots, G_{m_n};E_n)$ and $t \in \mathcal{L}(F;H)$, then
\begin{equation}\|t\circ A\circ(B_1,\ldots,B_n)\|_{\mathcal{H}}\le\|t\|\cdot\|A\|_{\mathcal{H}}\cdot
\|B_1\|\cdots\|B_n\|,\label{eqhi}\end{equation}
then $(\mathcal{H},\|\cdot\|_{\mathcal{H}})$ is called a \textit{$p$-normed hyper-ideal}. If all components $\mathcal{H}(E_1,\ldots, E_n;F)$ are complete spaces with respect to the topology generated by $\|\cdot\|_{\mathcal{H}}$, then  $(\mathcal{H},\|\cdot\|_{\mathcal{H}})$ is called a \textit{$p$-Banach hyper-ideal}. When $p=1$ we say that $(\mathcal{H},\|\cdot\|_{\mathcal{H}})$ is a \textit{normed hyper-ideal} or a \textit{Banach hyper-ideal}. If  $(\mathcal{H},\|\cdot\|_{\mathcal{H}})$ is a $p$-normed ($p$-Banach) hyper-ideal for some $p \in (0,1]$, then we say that it is a {\it quasi-normed hyper-ideal} ({\it quasi-Banach hyper-ideal}).\end{definition}

\begin{obs}\label{obshmi}\rm (i) It is plain that that every (normed, quasi-normed, Banach, quasi-Banach) hyper-ideal is a (normed, quasi-normed, Banach, quasi-Banach) multi-ideal. So, properties of multi-ideals are inherited by hyper ideals. For instance, if  $(\mathcal{H},\|\cdot\|_\mathcal{H})$ a $p$-normed hyper-ideal, then \begin{equation}\|\cdot\|\le\|\cdot\|_\mathcal{H}.\label{deshi}\end{equation}
 (ii) Any hyper-ideal is a normed hyper-ideal with the uniform norm $\|\cdot\|$. In the case it is a Banach hyper-ideal, we say that it is a {\it closed hyper-ideal}.\end{obs}

This section is devoted to the establishment of general proprieties of hyper-ideals, which are compared, case by case, with the corresponding property of multi-ideals. Illustrative examples are postponed to the next section.

As in the cases of operator ideals and multi-ideals, standard arguments give the:

\begin{prop}\label{fhi} Given a hyper-ideal $\mathcal{H}$, define
$$\overline{\mathcal{H}}(E_1,\ldots,E_n;F):=\overline{\mathcal{H}(E_1,\ldots,E_n;F)}^{\|\cdot\|},$$ for all $n\in \mathbb{N}$ and Banach spaces $E_1,\ldots,E_n,F$.
Then $(\overline{\mathcal{H}},\|\cdot\|)$ is the smallest closed hyper-ideal containing $\mathcal{H}$.\end{prop}

If $({\cal M}, \|\cdot\|_{\cal M})$ is a quasi-normed multi-ideal and $\varphi_1 \in E_1', \ldots, \varphi_n \in E_n', y \in F$, then $\varphi_1 \otimes \cdots \varphi_n \otimes y$ belongs to $\cal M$ and $\|\varphi_1 \otimes \cdots \varphi_n \otimes y\|_{\cal M}= \|\varphi_1 \otimes \cdots \varphi_n \otimes y\|$. Hyper-ideals enjoy a stronger property, which shall play an important role later (cf. Theorem \ref{hnmb}):

\begin{prop}\label{ndeshi}Let $(\mathcal{H},\|\cdot\|_\mathcal{H})$ be a quasi-normed hyper-ideal, $n$, $m_1<\cdots<m_n$ be natural numbers, $T_1\in\mathcal{L}(E_1,\ldots,E_{m_1}),\ldots, T_n \in \mathcal{L}(E_{m_{n-1}+1},\ldots,E_{m_n})$ and $y\in F$. Considered the $m_n$-linear operator $T_1\otimes\cdots \otimes T_n\otimes y \colon E_1 \times \cdots \times E_{m_n} \longrightarrow F$ defined by
$$T_1\otimes\cdots \otimes T_n\otimes y(x_1,\ldots,x_{m_n})=T_1(x_1,\ldots,x_{m_1})\cdots T_l(x_{m_{n-1}+1},\ldots,x_{m_n})\cdot y.$$ Then $T_1\otimes\cdots \otimes T_n\otimes y\in\mathcal{H}(E_1,\ldots,E_{m_n};F)$ and
$$\|T_1\otimes\cdots \otimes T_n\otimes y\|_\mathcal{H}=\|T_1\otimes\cdots\otimes T_n\otimes y\|=\|T_1\|\cdots\|T_n\|\cdot\|y\|.$$\end{prop}

\begin{proof} Considering the linear operator $1\otimes y \colon \mathbb{K}\longrightarrow F$ given by $1\otimes y(\lambda)=\lambda\cdot y$, we have  $$(1\otimes y)\circ I_{m_n}\circ (T_1,\ldots,T_n)=T_1\otimes\cdots\otimes T_n\otimes y.$$
As $I_{m_n}\in\mathcal{H}(^{m_n}\mathbb{K};\mathbb{K})$, from the hyper-ideal property of $\cal H$ we conclude that $T_1\otimes\cdots\otimes T_n\otimes y\in\mathcal{H}(E_1,\ldots,E_{m_n};F)$. Using first (\ref{deshi}) and then (\ref{eqhi}), we get
\begin{eqnarray*}\|T_1\otimes\cdots\otimes T_n\otimes y\|&\le& \|T_1\otimes\cdots\otimes T_n\otimes y\|_\mathcal{H}\\&=&\|(1\otimes y)\circ I_{m_n}\circ(T_1,\ldots,T_n)\|_\mathcal{H}\\&\le& \|(1\otimes y)\|\cdot\| I_{m_n}\|_\mathcal{H}\cdot\|T_1\|\cdots\|T_n\|\\
&=&\|T_1\|\cdots\|T_n\|\cdot\|y\|=\|T_1\otimes\cdots\otimes T_n\otimes y\|,\end{eqnarray*}
from which the desired equality follows.\end{proof}

The series criterion for operator ideals (\cite[9.4]{klauslivro}, \cite[6.2.3]{pietschlivro}) and for multi-ideals \cite[Satz 2.2.4]{andreas} works, {\it mutatis mutandis}, for hyper-ideals:

\begin{teo}[Series criterion]\label{cs} Let $0<p\le 1$  and $\mathcal{H}$ be a subclass of the class of all continuous multilinear operators between Banach spaces endowed with a map $\|\cdot\|_\mathcal{H}\colon\mathcal{H}\longrightarrow [0,+\infty)$. Then $(\mathcal{H},\|\cdot\|_\mathcal{H})$ is a $p$-Banach hyper-ideal if and only if the following conditions are satisfied:\\
{\rm (i)} $I_n\in\mathcal{H}(\mathbb{K}^n;\mathbb{K})$ and $\|I_n\|_{\mathcal{H}}=1$ for every $n \in \mathbb{N}$;\\
{\rm (ii)} If $(A_j)_{j=1}^\infty\subseteq\mathcal{H}(E_1,\ldots,E_n;F)$ is such that $\sum\limits_{j=1}^\infty\|A_j\|_\mathcal{H}^{p}<\infty$, then $$A:=\sum\limits_{j=1}^\infty A_j\in\mathcal{H}(E_1,\ldots,E_n;F)\ \mbox{and}\ \|A\|_\mathcal{H}^{p}\le\sum\limits_{j=1}^\infty \|A_j\|_\mathcal{H}^{p};$$
{\rm (iii)} If $n \in \mathbb{N}$, $1\le m_1<\cdots<m_n$, $G_1,\ldots,G_{m_n}$, $E_1,\ldots,E_n, F, H$ are Banach spaces, $B_1\in \mathcal{L}(G_1,\ldots, G_{m_1};E_1), \ldots, B_n\in \mathcal{L}(G_{m_{n-1}+1},\ldots, G_{m_n};E_n)$, $A \in \mathcal{H}(E_1,\ldots, E_n;F)$ and $t \in \mathcal{L}(F;H)$, then
$t\circ A\circ(B_1,\ldots,B_n)\in \mathcal{H}(G_1,\ldots, G_{m_n};H)$ and $$\|t\circ A\circ(B_1,\ldots,B_n)\|_{\mathcal{H}}\le\|t\|\cdot\|A\|_{\mathcal{H}}\cdot\|B_1\|\cdots\|B_n\|.$$\end{teo}

We are not aware of any easy-to-find reference where the basic properties of multi-ideals are proved in detail. Some of them can be found, in German, in the dissertations \cite{andreas, geiss}. For properties of hyper-ideals inherited from multi-ideals, when the argument is very similar to the case of operator ideals, such as (\ref{deshi}), Proposition \ref{fhi} and Theorem \ref{cs}, we omit the proofs. As to the next property, we think it is worth giving at least a sketch of the proof.

\begin{prop}\label{propdeshi}Let $0 < p \leq 1$ and $(\mathcal{G},\|\cdot\|_\mathcal{G})$ and $(\mathcal{H},\|\cdot\|_\mathcal{H})$ be $p$-Banach hyper-ideals such that $\mathcal{G}\subseteq\mathcal{H}$. Then, for every $n\in \mathbb{N}$ there is an constant $C_n$, depending only of $n$, such that $$\|A\|_\mathcal{H}\le C_n\|A\|_\mathcal{G},$$
for all Banach spaces $E_1,\ldots,E_n,F$ and $A\in\mathcal{G}(E_1,\ldots,E_n; F)$.\end{prop}

\begin{proof}Suppose, by contradiction, that there exists $m\in\mathbb{N}$ for which no such constant $C_m$ exists. This means that, for each $n\in\mathbb{N}$, there are Banach spaces $E_1^{(n)}, \ldots, E_m^{(n)},G_n$ and an operator $A_n\in\mathcal{G}(E_1^{(n)},\ldots,E_m^{(n)};G_n)$ such that $\|A_n\|_{\mathcal{H}}>2^{n/p}n\|A_n\|_{\mathcal{G}}$. It is not difficult to see that we can assume, wlog, that $\|A_n\|_{\mathcal{G}}=\dfrac{1}{2^{n/p}}$. So
$\|A_n\|_{\mathcal{H}}> n$ for every $n$. Consider the Banach spaces
  $E_i=\ap\bigoplus\limits_{n=1}^\infty E_i^{(n)}\fp_1$, $i=1,\ldots,m$, and $G=\ap\bigoplus\limits_{n=1}^\infty G_n\fp_\infty$. Now let $\pi_{in}\colon E_i\longrightarrow E_i^{(n)}$ and $\iota_n \colon G_n\longrightarrow G$ be the corresponding canonical projections and inclusions, respectively. Calling on the hyper-ideal property we get $\iota_n\circ A\circ(\pi_{1n},\ldots,\pi_{mn})\in\mathcal{G}(E_1,\ldots,E_m;G)$ for every $n$ and \begin{align*}\sum\limits_{n=1}^\infty\|\iota_n\circ A\circ(\pi_{1n},\ldots,\pi_{mn})\|_{\mathcal{G}}^{p}&\le\sum\limits_{n=1}^\infty
\|\iota_n\|^{p}\cdot \|A\|_{\mathcal{G}}^{p}\cdot \|\pi_{1n}\|^{p}\cdots \|\pi_{mn}\|^{p}= \sum\limits_{n=1}^\infty\dfrac{1}{2^n}<\infty.
\end{align*} As $(\mathcal{G},\|\cdot\|_{\mathcal{G}})$ is $p$-Banach hyper-ideal, by Theorem \ref{cs} and from the assumptions it follows that
$$A=\sum\limits_{n=1}^\infty \iota_n\circ A\circ(\pi_{1n},\ldots,\pi_{mn})\in\mathcal{G}(E_1,\ldots,E_m;G)
\subseteq\mathcal{H}(E_1,\ldots,E_m;G).$$ On other hand, calling $\iota_{in}\colon E_i^{(n)}\longrightarrow E_i$, $i=1,\ldots,m$, and $\pi_n \colon G\longrightarrow G_n$ the corresponding canonical projections and inclusions, we have $A_n=\pi_n\circ A\circ(\iota_{1n},\ldots,\iota_{mn})$ and $$n < \|A_n\|_{\mathcal{H}}=\|\pi_n\circ A\circ(\iota_{1n},\ldots,\iota_{mn})\|_{\mathcal{H}}\le\|\pi_n\|
\cdot\|A\|_{\mathcal{H}}\cdot \|\iota_{1n}\|\cdots\|\iota_{mn}\| = \|A\|_{\mathcal{H}},$$
for every $n$, a contradiction that completes the proof.\end{proof}

\section{Distinguished examples and the smallest Banach hyper-ideal}\label{action}
In this section we provide a number of illustrative examples of hyper-ideals, as well as important examples of multi-ideals that fail to be hyper-ideals. For each important non-hyper-ideal multi-ideal $\cal M$ we exhibit a hyper-ideal $\cal H$ that generalizes the same linear operator ideal, that is ${\cal M}^1 = {\cal H}^1$. This is the case of the class of nuclear multilinear operators, whose solution in the context of hyper-ideals gives a full description of the smallest Banach hyper-ideal (cf. Theorem \ref{hnmb}). The examples and results given in this section, together with the ones given in Section 3, make clear that, whenever the theory of multi-ideals is helpless, the theory of hyper-ideals is rich enough to provide its own solution.

\begin{ex}\label{ex1}\rm Exposing a deep distinction with the theory of multi-ideals, the purpose of this example is to show that the class ${\cal L}_f$ of the finite type multilinear operators is not a hyper-ideal. To do so, consider any scalar-valued multilinear operator that fails to be of finite type, for instance the following bilinear operator:
$$T\colon \ell_2 \times \ell_2 \longrightarrow \mathbb{K}~,~ T\left((x_i)_{i=1}^\infty,(y_i)_{i=1}^\infty\right)=\sum\limits_{i=1}^\infty x_iy_i.$$
Supposing that $\mathcal{L}_{f}$ is a hyper-ideal, as $Id_{\mathbb{K}} \in\mathcal{L}_{f}(\mathbb{K};\mathbb{K})$, the hyper-ideal property would give $T=Id_{\mathbb{K}}\circ T \in \mathcal{L}_{f}(^2\ell_2;\mathbb{K})$. This contradiction shows that ${\cal L}_f$  fails to be a hyper-ideal. As a matter of fact, the bilinear operator $T$ is not approximable by finite type operators, so the same reasoning shows that the multi-ideal $\overline{{\cal L}_f}$ of multilinear operators that can be approximated, in the uniform norm, by finite type operators is not a hyper-ideal either.\end{ex}

The class $\mathcal{L}_{f}$ is a distinguished multi-ideal in the sense that it is the smallest multi-ideal. As it is not a hyper-ideal, a smallest hyper-ideal is needed. Moreover its linear component should be the linear component of $\mathcal{L}_{f}$, that is, the operator ideal of finite rank operators. The example above gives us the hint:

\begin{teo}\label{mhi} The class $\mathcal{L}_\mathcal{F}$ of finite rank multilinear operators with the uniform norm is the smallest normed hyper-ideal. This means that, if $(\mathcal{H},\|\cdot\|_\mathcal{H})$ is a normed hyper-ideal, then $\mathcal{L}_\mathcal{F}\subseteq\mathcal{H}$ and $\|\cdot\|\le\|\cdot\|_\mathcal{H}$.\end{teo}
\begin{proof} Condition \ref{dhi}(1) is straightforwardly checked for $\mathcal{L}_\mathcal{F}$. For the hyper-ideal property, let $t \in \mathcal{L}(F;H)$, $A \in \mathcal{\mathcal{L}_{\mathcal{F}}}(E_1,\ldots, E_n;F), B_1\in \mathcal{L}(G_1,\ldots, G_{m_1};E_1), \ldots,$ and $B_n\in \mathcal{L}(G_{m_{n-1}+1},\ldots, G_{m_n};E_n)$  be given. By $R(C)$ we mean the range of the map $C$. Then ${\rm span}R(A)$ is a finite dimensional subspace of $F$, and since $t$ is linear, ${\rm span}R(t \circ A)$ is a finite dimensional subspace of $H$, so
$${\rm span} R(t \circ A \circ (B_1, \ldots, B_n)) \subseteq {\rm span}R(t\circ A),$$
what yields that $t \circ A \circ (B_1, \ldots, B_n)$ is of finite rank, proving that $\mathcal{L}_\mathcal{F}$ is a hyper-ideal.

Now let $(\mathcal{H},\|\cdot\|_\mathcal{H})$ be a normed hyper-ideal and $B \in \mathcal{L}(E_1,\ldots, E_n)$ and $y \in F$ be given. Considering again the linear operator $1\otimes y\colon \mathbb{K} \longrightarrow F$ defined by $1\otimes y(\lambda)=\lambda y$, we have $1\otimes y \in \mathcal{L}_{f}(\mathbb{K};F) \subseteq \mathcal{H}(\mathbb{K};F)$. From the hyper-ideal property it follows that $$B\otimes y=(1\otimes y)\circ B\in\mathcal{H}(E_1,\ldots, E_n;F).$$
As any finite rank $n$-linear operator is a linear combination of operators of the form $B \otimes y$ and $\mathcal{H}(E_1,\ldots, E_n;F)$ is a linear space, it follows that $\mathcal{H}(E_1,\ldots, E_n;F)$ contains $\mathcal{L}_{\mathcal{F}}(E_1,\ldots, E_n;F)$.
The norm inequality follows from  (\ref{deshi}).\end{proof}

A deep distinction with the theory of multi-ideals is now clear, namely, in Definition \ref{dhi}, the containment of the finite type operators can be equivalently replaced by the containment of the finite rank operators:

\begin{cor}\label{difmulti} Let $\mathcal{H}$ be a class of continuous multilinear operators fulfilling the hyper-ideal property {\rm \ref{dhi}(2)} such that each component $\mathcal{H}(E_1,\ldots, E_n;F)$ is a linear subspace of $\mathcal{L}(E_1,\ldots, E_n;F)$. Then ${\cal L}_f \subseteq \cal H$ if and only if ${\cal L}_{\cal F} \subseteq \cal H$.\end{cor}

It is well known that the class $\cal N$ of nuclear linear operators is the smallest Banach operator ideal and the class ${\cal L}_{\cal N}$ of nuclear multilinear mappings is the smallest Banach multi-ideal. So ${\cal L}_{\cal N}$ is our first candidate to be the smallest Banach hyper-ideal. Next we see that this is not the case. The study of nuclear nonlinear operators goes back to Gupta \cite{gupta}. In the multilinear setting, an $n$-linear operator $A\in\mathcal{L}(E_1,\ldots,E_n;F)$ is called {\it nuclear} if, for each $l = 1,\ldots,n$, we can find a bounded sequence $(\varphi_j^{(l)})_{j=1}^\infty$ of linear functionals in $E_l'$, a sequence $(\lambda_j)_{j=1}^\infty\in\ell_1$ and a bounded sequence $(y_j)_{j=1}^\infty$ in $F$ such that $$A(x_1,\ldots,x_n)=\sum\limits_{j=1}^\infty \lambda_j\varphi_j^{(1)}(x_1)\cdots\varphi_j^{(n)}(x_n)\cdot y_j$$
for all $x_1 \in E_1, \ldots, x_n \in E_n$. The expression above is called a \textit{nuclear representation} for $A$ and the space of all nuclear $n$-linear operators from $E_1 \times \cdots \times E_n$ to $F$ is denoted by  $ \mathcal{L}_\mathcal{N}(E_1,\ldots,E_n;F)$. The nuclear norm
$\|\cdot\|_{\mathcal{L}_\mathcal{N}}\colon \mathcal{L}_\mathcal{N}\longrightarrow[0,\infty)$ is defined by $$\|A\|_{\mathcal{L}_\mathcal{N}}=\inf\left\{\sum\limits_{j=1}^\infty |\lambda_j|\cdot\|\varphi_j^{(1)}\|\cdots\|\varphi_j^{(n)}\|\cdot\|y_j\|\right\},$$ where the infimum is taken over all nuclear representations of $A$. As mentioned before, $({\cal L}_{\cal N}, \|\cdot \|_{{\cal L}_{\cal N}})$ is the smallest Banach multi-ideal. Its linear components recover the classical ideal of nuclear operators.

\begin{ex}\rm \label{nucl} Let us see that the class $\mathcal{L}_\mathcal{N}$ of nuclear multilinear operators fails to be a hyper-ideal. Indeed, considering the partial sums of a nuclear representation of a nuclear multilinear mapping, it is not difficult to check that $\mathcal{L}_\mathcal{N}\subseteq \overline{{\cal L}_f}$ In Example \ref{ex1} we saw that $\overline{{\cal L}_f}$ does not contain ${\cal L}_{\cal F}$, so $\mathcal{L}_\mathcal{N}$ does not contain ${\cal L}_{\cal F}$ either. As  $\mathcal{L}_\mathcal{N}$ is a multi-ideal, from Corollary \ref{difmulti} we conclude that $\mathcal{L}_\mathcal{N}$ does not fulfil the hyper-ideal property.\end{ex}

Once the multi-ideal of nuclear multilinear operators is out of the game, we need another class to play the role of the smallest Banach hyper-ideal. Thinking about the linear components, this new class should be a multilinear generalization of the ideal of nuclear linear operators different from ${\cal L}_{\cal N}$. Having this in mind, and inspired by the several nuclear-type multilinear operators introduced in Matos \cite{matos} and Popa \cite{popan}, we consider the following classes of multilinear operators (we assume $1/\infty = 0$).

\begin{definition}\rm Let $s\in(0,\infty)$ and $r\in [1,\infty]$ be such that $1\le1/s+1/r$. An $n$-linear continuous operator $A \in \mathcal{L}(E_1,\ldots,E_n;F)$ is called \textit{hyper-$(s,r)$-nuclear} if there are sequences $(\lambda_j)_{j=1}^\infty\in\ell_s$, $(T_j)_{j=1}^\infty\in\ell_r^w(\mathcal{L}(E_1,\ldots,E_n))$ and $(y_j)_{j=1}^\infty\in\ell_\infty(F)$ such that \begin{equation}\label{eq:hn}A(x_1,\ldots,x_n)=\sum\limits_{j=1}^\infty\lambda_jT_j\otimes y_j(x_1,\ldots,x_n)=\sum\limits_{j=1}^\infty\lambda_jT_j(x_1,\ldots,x_n) y_j,\end{equation} for all $(x_1,\ldots,x_n)\in E_1\times\cdots\times E_n$. In this case we write $A\in\mathcal{L}_{\mathcal{HN}_(s,r)}(E_1,\ldots,E_n;F)$. The {\it hyper-$(s,r)$-nuclear norm} $\|\cdot\|_{\mathcal{L}_{\mathcal{HN}_{(s,r)}}}\colon \mathcal{L}_{\mathcal{HN}_{(s,r)}}\longrightarrow [0,\infty)$ is defined by $$\|A\|_{\mathcal{L}_{\mathcal{HN}_{(s,r)}}}=\inf\{\|(\lambda_j)_{j=1}^\infty\|_s\cdot \|(T_j)_{j=1}^\infty\|_{w,r}\cdot\|(y_j)_{j=1}^\infty\|_\infty\},$$ where the infimun is taken over all representations of $A$ as in (\ref{eq:hn}).\end{definition}

It is easy to prove that the series in (\ref{eq:hn}) is absolutely convergent, hence convergent.

In the case $s=1$ and $r=\infty$ we have $(\lambda_j)_{j=1}^\infty\in\ell_1$ and that the sequence $(T_j)_{j=1}^\infty$ is bounded (remember that $\ell_\infty^w(E)=\ell_\infty(E)$ for every Banach space $E$). Thus every nuclear multilinear operator, in the sense of Example \ref{nucl}, is  hyper-$(1,\infty)$-nuclear. For this reason, hyper-$(1,\infty)$-nuclear operators are simply called \textit{hyper-nuclear} operators and we write $\mathcal{L}_{\mathcal{HN}}$ instead of $\mathcal{L}_{\mathcal{HN}_{(1,\infty)}}$. It is clear that the linear component of $\mathcal{L}_{\mathcal{HN}}$ coincides with the ideal of nuclear operators. Now it is clear that $\mathcal{L}_{\mathcal{HN}}$ is a good candidate to be the smallest Banach hyper-ideal. We treat the general case before going into this issue:

\begin{teo}\label{hnhi}The class $(\mathcal{L}_{\mathcal{HN}_{(s,r)}}, \|\cdot\|_{\mathcal{L}_{\mathcal{HN}_{(s,r)}}})$ of hyper-$(s,r)$-nuclear
multilinear operators is a $p$-Banach hyper-ideal, where $\frac{1}{p}=\frac{1}{s}+\frac{1}{r}$.\end{teo}

\begin{proof} Let us prove that the conditions in Theorem \ref{cs} are satisfied:\\
 (i) It is plain that $I_n\in\mathcal{L}_{\mathcal{HN}_{(s,r)}}(\mathbb{K}^n;\mathbb{K})$. Regarding $I_n$ as a representation of itself it follows that $\|I_n\|_{\mathcal{L}_{\mathcal{HN}_{(s,r)}}}\le1$. Assuming that $\|I_n\|_{\mathcal{L}_{\mathcal{HN}_{(s,r)}}}<1$, there would exist a representation $\sum\limits_{j=1}^\infty\lambda_j\otimes T_j$ of $I_n$ with $\|(\lambda_j)_{j=1}^\infty\|_{s}\cdot\|(T_j)_{j=1}^\infty\|_{w,r}<1$. But $$1=|I_n(e_1,\ldots,e_1)|\le\sum\limits_{j=1}^\infty|\lambda_j|\cdot\|T_j\|
\cdot\|e_1\|^n\le \|(\lambda_j)_{j=1}^\infty\|_{s}\cdot\|(T_j)_{j=1}^\infty\|_{w,r}<1,$$
a contradiction that gives $\|I_n\|_{\mathcal{L}_{\mathcal{HN}_{(s,r)}}}=1$.\\
(ii) Let $(A_j)_{j=1}^\infty\subseteq\mathcal{L}_{\mathcal{HN}_{(s,r)}}(E_1,\ldots,E_n;F)$ be such that $\sum\limits_{j=1}^\infty\|A_j\|_{\mathcal{L}_{\mathcal{HN}_{(s,r)}}}^p<\infty.$ Given $\varepsilon>0$, for every $j\in\mathbb{N}$ there are sequences $(\lambda_{jk})_{k=1}^\infty\in\ell_s$, $(T_{jk})_{k=1}^\infty\in\ell_r^w(\mathcal{L}(E_1,\ldots,E_n))$ and $(y_{jk})_{k=1}^\infty\in\ell_\infty(F)$ such that $A_j=\sum\limits_{k=1}^\infty\lambda_{jk} T_{jk}\otimes y_{jk}$  and $$\|(\lambda_{jk})_{k=1}^\infty\|_s\cdot\|(T_{jk})_{k=1}^\infty\|_{w,r}\cdot
\|(y_{jk})_{k=1}^\infty\|_\infty<(1+\varepsilon)\|A_j\|_{\mathcal{L}_{\mathcal{HN}_{(s,r)}}}.$$
Wlog, we can assume, for each $j$, that $\|(y_{jk})_{k=1}^\infty\|_\infty=1$ and $$\|(\lambda_{jk})_{k=1}^\infty\|_s<\ap(1+\varepsilon)\|A_j\|_{\mathcal{L}_{\mathcal{HN}_{(s,r)}}}
\fp^{p/s},\ \|(T_{jk})_{k=1}^\infty\|_{w,r}<\ap(1+\varepsilon)\|A_j\|_{\mathcal{L}_{\mathcal{HN}_{(s,r)}}}
\fp^{p/r}.$$ From \begin{eqnarray}\sum\limits_{j,k=1}^\infty|\lambda_{jk}|^s=\sum\limits_{j=1}^\infty
\sum\limits_{k=1}^\infty|\lambda_{jk}|^s=
\sum\limits_{j=1}^\infty\|(\lambda_{jk})_{k=1}^\infty\|_s^s <(1+\varepsilon)^p\cdot\sum\limits_{j=1}^\infty\|A_j\|_{\mathcal{L}_{\mathcal{HN}_{(s,r)}}}^p< \infty,\label{eq:s}\end{eqnarray} we conclude that $(\lambda_{jk})_{j,k=1}^\infty\in\ell_s$. For each linear functional $\varphi\in(\mathcal{L}(E_1,\ldots,E_n))'$ with $\|\varphi\| \leq 1$ we have \begin{align}\sum\limits_{j,k=1}^\infty|\varphi(T_{jk})|^r
\leq\sum\limits_{j=1}^\infty\|(T_{jk})_{k=1}^\infty\|_{w,r}^r < (1+\varepsilon)^p\cdot \sum\limits_{j=1}^\infty\|A_j\|_{\mathcal{L}_{\mathcal{HN}_{(s,r)}}}^p<\infty.\label{eq:r}
\end{align} Therefore $(T_{jk})_{j,k=1}^\infty\in\ell_{r}^w(\mathcal{L}(E_1,\ldots,E_n))$. We already know that $(y_{jk})_{j,k=1}^\infty\in\ell_\infty(F)$;  so, for all $x_1 \in E_1, \ldots, x_n \in E_n$, the series \begin{equation}\sum\limits_{j,k=1}^\infty\lambda_{jk} T_{jk}\otimes y_{jk}(x_1, \ldots, x_n) \label{eq:rep}\end{equation} is absolutely convergent in the Banach space $F$. Then,    \begin{align*}\sum\limits_{j,k=1}^\infty\lambda_{jk} T_{jk}\otimes y_{jk}(x_1, \ldots, x_n)&=\sum\limits_{j=1}^\infty\sum\limits_{k=1}^\infty\lambda_{jk} T_{jk}\otimes y_{jk}(x_1, \ldots, x_n)\\&=\sum\limits_{j=1}^\infty A_j(x_1, \ldots, x_n)=:A(x_1, \ldots, x_n)
\end{align*}
defines  $A \colon E_1 \times \cdots \times E_n \longrightarrow F$ and shows that  \eqref{eq:rep} it is a representation of $A$ as in (\ref{eq:hn}), proving that $A$ is hyper-$(s,r)$-nuclear.
As $\|(y_{jk})_{j,k=1}^\infty\|_\infty=1$ , from \eqref{eq:s} and \eqref{eq:r} we get \begin{eqnarray*}\|A\|_{\mathcal{L}_{\mathcal{HN}_{(s,r)}}}&\le& \|(\lambda_{jk})_{j,k=1}^\infty\|_s^p\cdot\|(T_{jk})_{j,k=1}^\infty\|_{w,r}^p\cdot\|(y_{jk})_{j,k=1}^\infty\|_\infty^p \\&\le&\ap(1+\varepsilon)^p\cdot\sum\limits_{j=1}^\infty\|A_j\|_{\mathcal{L}_{\mathcal{HN}_{(s,r)}}}^p
\fp^{p/s}\cdot \ap(1+\varepsilon)^p\cdot\sum\limits_{j=1}^\infty\|A_j\|_{\mathcal{L}_{\mathcal{HN}_{(s,r)}}}^p\fp^{p/r}\\
&=& (1+\varepsilon)^p\cdot\sum\limits_{j=1}^\infty\|A_j\|_{\mathcal{L}_{\mathcal{HN}_{(s,r)}}}^p.\end{eqnarray*}
Letting $\varepsilon\longrightarrow0$ we obtain the desired inequality.\\
(iii) Let $1\le m_1<\cdots<m_n$, $A \in \mathcal{L}_{\mathcal{HN}_{(s,r)}}(E_1,\ldots, E_n;F)$, $B_1\in \mathcal{L}(G_1,\ldots, G_{m_1};E_1),$ $\ldots, B_n\in \mathcal{L}(G_{m_{n-1}+1},\ldots, G_{m_n};E_n)$ and $t \in \mathcal{L}(F;H)$. We can write $A=\sum\limits_{j=1}^\infty\lambda_jT_j\otimes y_j$, with $(\lambda_j)_{j=1}^\infty\in\ell_s$, $(T_j)_{j=1}^\infty\in\ell_r^w(\mathcal{L}(E_1,\ldots,E_n))$ and $(y_j)_{j=1}^\infty\in \ell_\infty(F)$. Defining $S_j:=T_j\circ(B_1,\ldots,B_n)$ and $z_j=t(y_j)$ for every $j\in\mathbb{N}$, we have
\begin{align*}t\circ A\circ&(B_1,\ldots,B_n)(x_1,\ldots,x_{m_n})=
\sum\limits_{j=1}^\infty\lambda_jS_j\otimes z_j(x_1,\ldots,x_{m_n})\end{align*}
for all $x_1, \in E_1, \ldots, x_n \in E_n$.
    It is clear that $(z_j)_{j=1}^\infty\in\ell_\infty(H)$ and $$\|(z_j)_{j=1}^\infty\|_\infty\le\|t\|\cdot\|(y_j)_{j=1}^\infty\|_\infty.$$
Given a linear functional $\psi
\in\mathcal{L}(G_1,\ldots,G_{m_n})'$, considering the continuous linear operator $L_{(B_1,\ldots,B_n)}\colon \mathcal{L}(E_1,\ldots,E_{n})\longrightarrow\mathcal{L}(G_1,\ldots,G_{m_n})$ defined by $$L_{(B_1,\ldots,B_n)}(T)=T\circ(B_1,\ldots,B_n),$$
as $\psi\circ L_{(B_1,\ldots,B_n)}\in\mathcal{L}(E_1,\ldots,E_{n})'$ and $(T_j)_{j=1}^\infty\in\ell_r^w(\mathcal{L}(E_1,\ldots,E_{n}))$, we have \begin{eqnarray*}\sum\limits_{j=1}^\infty|\psi(S_j)|^r&=&
\sum\limits_{j=1}^\infty|\psi(T_j\circ(B_1,\ldots,B_n))|^r= \sum\limits_{j=1}^\infty|\psi(L_{(B_1,\ldots,B_n)}(T_j))|^r \\&=&\sum\limits_{j=1}^\infty|\psi\circ L_{(B_1,\ldots,B_n)}(T_j)|^r \le \|\psi \circ L_{(B_1,\ldots,B_n)}\|^r\cdot\|(T_j)_{j=1}^\infty\|_{w,r}^r <\infty,\end{eqnarray*} This shows that $(S_j)_{j=1}^\infty\in\ell_r^w(\mathcal{L}(G_1,\ldots,G_{m_n}))$, hence $\sum\limits_{j=1}^\infty\lambda_jS_j\otimes z_j$ is a representation of $t\circ A\circ(B_1,\ldots,B_n)$ as in (\ref{eq:hn}). So $t\circ A\circ(B_1,\ldots,B_n)$ is hyper-$(s,r)$-nuclear and
\begin{align*}\|t\circ A\circ(B_1,\ldots &,B_n)\|_{\mathcal{L}_{\mathcal{HN}_{(s,r)}}}\le
\|(\lambda_j)_{j=1}^\infty\|_{s}\cdot\|(S_j)_{j=1}^\infty\|_{w,r}\cdot\|(z_j)_{j=1}^\infty\|_\infty \\& \leq \|(\lambda_j)_{j=1}^\infty\|_{s}\cdot\|L_{(B_1,\ldots,B_n)}
\|\cdot\|(T_j)_{j=1}^\infty\|_{w,r}\cdot\|(z_j)_{j=1}^\infty\|_\infty
\\& \leq  \|(\lambda_j)_{j=1}^\infty\|_{s}\cdot\|B_1\| \cdots \|B_n\|
\cdot\|(T_j)_{j=1}^\infty\|_{w,r}\cdot\|(z_j)_{j=1}^\infty\|_\infty
\\&\le \|t\|\cdot\left[\|(\lambda_j)_{j=1}^\infty\|_{s}\cdot
\|(T_j)_{j=1}^\infty\|_{w,r}\cdot\|(y_j)_{j=1}^\infty\|_\infty\right]\cdot\|B_1\|\cdots\|B_n\|.
\end{align*}Taking the infimum over all hyper-$(s,r)$-nuclear representations of $A$ we have $$\|t\circ A\circ(B_1,\ldots,B_n)\|_{\mathcal{L}_{\mathcal{HN}_{(s,r)}}}\le \|t\|\cdot\|A\|_{\mathcal{L}_{\mathcal{HN}_{(s,r)}}}\cdot\|B_1\|\cdots\|B_n\|.$$\end{proof}

\begin{cor} The class $({\cal L}_{\cal HN}, \|\cdot\|_{{\cal L}_{\cal HN}})$ of hyper-nuclear multilinear operators is a Banach hyper-ideal and  ${\cal L}_{\cal N} \subsetneqq{\cal L}_{\cal HN}$.\end{cor}
\begin{proof} Make $s = 1$ and $r = \infty$ in Theorem \ref{hnhi} to obtain the first assertion. The inclusion is obvious and the classes are different because ${\cal L}_{\cal N}$ is not a hyper-ideal (Example \ref{nucl}).\end{proof}

To show that ${\cal L}_{\cal HN}$ is the smallest Banach hyper-ideal we need the following characterization of hyper-nuclear norm:

\begin{lema}\label{lemahn}For every operator $A\in\mathcal{L}_{\mathcal{HN}}(E_1,\ldots,E_n;F)$, $$\|A\|_{\mathcal{L}_{\mathcal{HN}}}=\inf\ach\sum\limits_{j=1}^\infty|\lambda_j|\cdot\|T_j\|\cdot\|y_j\|\fch,$$ where the infimum is taken over all representations of $A$ as in {\rm (\ref{eq:hn})}.\end{lema}

\begin{proof} Given a representation $\sum\limits_{j=1}^\infty \lambda_j T_j \otimes y_j$ of $A$ as in (\ref{eq:hn}), we can assume $T_j\neq0$ and $y_j\neq0$ for every $j\in\mathbb{N}$. Defining the sequences $(\mu_j)_{j=1}^\infty\subseteq\mathbb{K}$, $(S_j)_{j=1}^\infty\subseteq\mathcal{L}(E_1,\ldots,E_n)$ and $(z_j)_{j=1}^\infty\subseteq F$ by $$\mu_j=\lambda_j\|T_j\|\cdot\|y_j\|,\ S_j=\frac{T_j}{\|T_j\|}\ ~\mbox{and}~\ z_j=\frac{y_j}{\|y_j\|},$$ it is easy to check that $\sum\limits_{j=1}^\infty\mu_j S_j\otimes z_j$ is a representation of $A$ as in (\ref{eq:hn}) as well. Then $$\|A\|_{\mathcal{L}_{\mathcal{HN}}} \leq \|(\mu_j)_{j=1}^\infty\|_{1}\cdot\|(S_j)_{j=1}^\infty\|_{\infty}\cdot\|(z_j)_{j=1}^\infty\|_\infty= \sum\limits_{j=1}^\infty|\lambda_j|\cdot\|T_j\|\cdot\|y_j\|.$$ Taking the infimum of over all representations of $A$ we obtain one inequality. The reverse inequality follows immediately from the definition of the hyper-nuclear norm.\end{proof}

\begin{teo}\label{hnmb}The class $(\mathcal{L}_{\mathcal{HN}},\|\cdot\|_{\mathcal{L}_{\mathcal{HN}}})$ of hyper-nuclear multilinear operators is the smallest Banach hyper-ideal, in sense that if  $(\mathcal{H},\|\cdot\|_{\mathcal{H}})$ is a Banach hyper-ideal, then $\mathcal{L}_{\mathcal{HN}}\subseteq \mathcal{H}$ and $\|\cdot\|_{\mathcal{H}}\le\|\cdot\|_{\mathcal{L}_{\mathcal{HN}}}$.\end{teo}

\begin{proof}Let $(\mathcal{H},\|\cdot\|_\mathcal{H})$ be a Banach hyper-ideal. Let $A\in\mathcal{L}_{\mathcal{HN}}(E_1,\ldots,E_n;F)$ and let $\sum\limits_{j=1}^\infty\lambda_jT_j\otimes y_j$ be a representation of $A$ as in (\ref{eq:hn}). Given $\varepsilon>0$, let $k_0\in\mathbb{N}$ be such that \begin{equation}\sum\limits_{j=k}^\infty|\lambda_j|\cdot\|T_j\|\cdot\|y_j\|<
\varepsilon\label{eq:estima}\end{equation}
for every $k\ge k_0$. As $\cal H$ is a hyper-ideal, calling on Theorem \ref{mhi} we have
$$B_k:=\sum\limits_{j=1}^k\lambda_jT_j\otimes y_j \in \mathcal{L}_{\cal F}(E_1,\ldots,E_n;F) \subseteq \mathcal{H}(E_1,\ldots,E_n;F)$$
for every $k$. For all $k>i\ge k_0$, Proposition \ref{ndeshi} gives $$\|B_k-B_i\|_\mathcal{H}=\an\sum\limits_{j=i+1}^k\lambda_jT_j\otimes y_j\fn_\mathcal{H}\le \sum\limits_{j=i+1}^k|\lambda_j|\cdot\|T_j\otimes y_j\|_\mathcal{H}=\sum\limits_{j=i+1}^k|\lambda_j|\cdot\|T_j\|\cdot\|y_j\|<\varepsilon,$$ showing that $(B_k)_{k=1}^\infty$ is a Cauchy sequence in the Banach space $\mathcal{H}(E_1,\ldots,E_n;F)$. Then there is $B\in\mathcal{H}(E_1,\ldots,E_n;F)$ such that $B_k \stackrel{\|\cdot\|_{\cal H}}{\longrightarrow} B$. From (\ref{deshi}), $\|\cdot\|\le\|\cdot\|_\mathcal{H}$, so $B_k \stackrel{\|\cdot\|}{\longrightarrow} B$. By (\ref{eq:estima}) it follows easily that $B_k \stackrel{\|\cdot\|}{\longrightarrow} A$, thus $A=B$. Hence $A\in\mathcal{H}(E_1,\ldots,E_n;F)$, proving that $\mathcal{L}_{\mathcal{HN}}\subseteq\mathcal{H}$.
Using Proposition \ref{ndeshi} once again,  $$\|A\|_{\mathcal{H}}\le \sum\limits_{j=1}^\infty\|\lambda_jT_j\otimes y_j\|_{\mathcal{H}}=\sum\limits_{j=1}^\infty|\lambda_j|\cdot \|T_j\otimes y_j\|_{\cal H}=\sum\limits_{j=1}^\infty|\lambda_j|\cdot \|T_j\|\cdot\| y_j\|.$$ Taking the infimum over all hyper-nuclear representations of $A$, from Lemma \ref{lemahn} it follows that $\|A\|_{\mathcal{H}}\le\|A\|_{\mathcal{L}_{\mathcal{HN}}}$.\end{proof}

\section{Composition ideals and multilinearly sequentially continuous operators}\label{ptp}

This section has a twofold purpose. First, using the notion of composition ideals, we show that many important multi-ideals are hyper-ideals, for example the classes of compact and weakly compact multilinear operators.  Second, we introduce the class of multilinearly sequentially continuous operators as a response, in the realm of hyper-ideals, to the fact that the multi-ideal of weakly sequentially continuous multilinear operators fails to be a hyper-ideal.

\begin{definition}[Composition ideals]\rm Given an operator ideal $\mathcal{I}$,  an $n$-linear operator $A \in \mathcal{L}(E_1,\ldots,E_n;F)$ belongs to the {\it composition ideal} ${\cal I} \circ {\cal L}$, in symbols,  $A \in \mathcal{I}\circ\mathcal{L}(E_1,\ldots,E_n;F)$, if there exist a Banach space $G$, a linear operator  $u\in\mathcal{I}(G;F)$ and an $n$-linear operator  $B \in \mathcal{L}(E_1,\ldots, E_n;G)$ such that $A=u\circ B.$ If $(\mathcal{I}, \|\cdot\|_{\cal I})$ is a $p$-normed operator ideal, $0 < p \leq 1$, we define $$\|A\|_{\mathcal{I}\circ\mathcal{L}}=\inf\{\|u\|_{\mathcal{I}}\cdot \|B\|\},$$ where the infimum is taken over all factorizations $A=u\circ B$ with $u$ belonging to $\mathcal{I}$.\label{idcomp}\end{definition}

It is well known that $({\cal I}\circ {\cal L}, \|\cdot\|_{{\cal I}\circ {\cal L}})$ is a $p$-normed ($p$-Banach) multi-ideal whenever $(\mathcal{I}, \|\cdot\|_{\cal I})$ is a $p$-normed ($p$-Banach) operator ideal. Details can be found in \cite{bpr2}.

\begin{teo}\label{hicom}If $\mathcal{I}$ is an operator ideal, then $\mathcal{I}\circ\mathcal{L}$ is a hyper-ideal. If $(\mathcal{I},\|\cdot\|_{\mathcal{I}})$ is a $p$-normed ($p$-Banach) operator ideal, then $(\mathcal{I}\circ\mathcal{L},\|\cdot\|_{\mathcal{I}\circ\mathcal{L}})$ is a $p$-normed ($p$-Banach) hyper-ideal. In particular, if $\mathcal{I}$ is a closed operator ideal, then $\mathcal{I}\circ\mathcal{L}$ is a closed hyper-ideal.\end{teo}

\begin{proof} As we know that $\mathcal{I}\circ\mathcal{L}$ is a ($p$-normed, $p$-Banach) multi-ideal, all that is left to be checked is the hyper-ideal property.
Let $B_1\in \mathcal{L}(G_1,\ldots, G_{m_1};E_1), \ldots,\ B_n\in \mathcal{L}(G_{m_{n-1}+1},\ldots, G_{m_n};E_n)$, where $m_1<\cdots<m_n$,
$A \in \mathcal{I}\circ\mathcal{L}(E_1,\ldots, E_n;F)$ and $t \in \mathcal{L}(F;H)$. We can write $A=v\circ C$, where $C\in \mathcal{L}(E_1,\ldots, E_n;F_1)$ and $v\in \mathcal{I}(F_1;F)$. Defining $u:=t\circ v \in \mathcal{I}(F_1;H)$ and $D:=C\circ(B_1,\ldots,B_n)\in \mathcal{L}(G_1,\ldots,G_{m_n};F_1)$, it follows that
$$t\circ A\circ (B_1,\ldots,B_n)=t\circ(v\circ C)\circ(B_1,\ldots,B_n)=(t\circ v)\circ (C\circ(B_1,\ldots,B_n))=u\circ D,$$
which proves that $t\circ A\circ (B_1,\ldots,B_n)\in \mathcal{I}\circ\mathcal{L}(G_1,\ldots,G_{m_n};H).$
Furthermore, \begin{eqnarray*}\|t\circ A\circ(B_1,\ldots,B_n)\|_{\mathcal{I}\circ\mathcal{L}}&=& \|(t\circ v)\circ (C\circ(B_1,\ldots,B_n))\|_{\mathcal{I}\circ\mathcal{L}} \\&\le&\|t\circ v\|_{\mathcal{I}}\cdot \|C\circ(B_1,\ldots,B_n)\|\\&\le& \|t\|\cdot\left[\|v\|_{\mathcal{I}}\cdot \|C\|\right]\cdot\|B_1\|\cdots\|B_n\|.\end{eqnarray*}
Taking the infimum over all possible factorizations we get $$\|t\circ A\circ(B_1,\ldots,B_n)\|_{\mathcal{I}\circ\mathcal{L}}\le \|t\|\cdot\|A\|_{\mathcal{I}\circ\mathcal{L}}\cdot\|B_1\|\cdots\|B_n\|.$$

The last assertion follows from the former ones and \cite[Corollary 3.8]{bpr2}.\end{proof}

\begin{ex}\rm A multilinear operator $A \in \mathcal{L}(E_1,\ldots,E_n;F)$ is said to be {\it compact} ({\it weakly compact}) if $A(B_{E_1}, \ldots, B_{E_n})$ is a relatively compact (relatively weakly compact) subset of $F$. By $\cal K$ and $\cal W$ we denote the closed ideals of compact and weakly compact operators, and by ${\cal L}_{\cal K}$ and ${\cal L}_{\cal W}$ the classes of compact and weakly compact multilinear operators.  Pe{\l}czy\'nski \cite{pelczynski} showed that $$\mathcal{L}_{\mathcal{K}}=\mathcal{K}\circ\mathcal{L} {\rm ~~ and~~} \mathcal{L}_{\mathcal{W}}=\mathcal{W}\circ\mathcal{L}.$$
By Theorem \ref{hicom} the classes of compact and weakly compact multilinear operators are closed hyper-ideals.\end{ex}

Let us take a look at another well studied closed multi-ideal (see, e.g., \cite{arondimant,boydryan,cgg}):

\begin{ex}\label{Exwsc}\rm A multilinear operator $A\in \mathcal{L}(E_1,\ldots,E_n;F)$ is said to be {\it
weakly sequentially continuous} if $A(x_j^{1},\ldots,x_j^{n})\longrightarrow A(x_1,\ldots,x_n)$ in the norm of $F$ whenever $(x_j^{l})_{j=1}^\infty$ is weakly convergent to $x_l$ in $E_l$, $l=1,\ldots,n$. The class of
weakly sequently continuous operators is denoted by $\mathcal{L}_{wsc}$.  As continuous linear operators are weak-to-weak continuous,
it is immediate that $\mathcal{L}_{wsc}$ is a multi-ideal. It is not difficult to prove that it is a closed multi-ideal.  On the other hand, let us see that $\mathcal{L}_{wsc}$ is not a hyper-ideal. In fact, as the sequence  $(e_i)_{i=1}^\infty$ of canonical unit vectors is weakly null in $\ell_2$ and non-norm null in $\ell_1$, the continuous bilinear operator
$$A \colon \ell_2\times\ell_2\longrightarrow \ell_1~,~A((x_j)_{j=1}^\infty,(y_j)_{j=1}^\infty)=(x_jy_j)_{j=1}^\infty,$$
fails to be weakly sequentially continuous. If $\mathcal{L}_{wsc}$ were a hyper-ideal, by the Schur property of $\ell_1$ we would have $A=Id_{\ell_1}\circ A\in\mathcal{L}_{wsc}(^2\ell_2;\ell_1).$ This contradiction shows that $\mathcal{L}_{wsc}$ fails the hyper-ideal property.\end{ex}

Now we need a closed hyper-ideal to play the role $\mathcal{L}_{wsc}$ plays the in context of multi-ideals. Let $\cal CC$ be the closed ideal of completely continuous linear operators (weakly convergent sequences are sent to norm convergent ones). If we just think of a closed hyper-ideal $\cal H$ whose linear component is $\cal CC$, that is ${\cal H}^1 = {\cal CC}$, then the composition ideal ${\cal CC} \circ {\cal L}$ arises as a natural candidate. Let us stress that ${\cal CC} \circ {\cal L}$ is {\it not} a good solution. First, the definition of multilinear operators belonging to ${\cal CC} \circ {\cal L}$ is not a multilinear analogue of the definition of linear operators belonging to $\cal CC$. Second, ${\cal CC} \circ {\cal L}$ seems to be  rather {\it small}. For example, coincidence situations are extremely rare for ${\cal CC} \circ {\cal L}$: reasoning as in \cite[Proposition 3.3]{bp} we have
\begin{center} ${\cal CC} \circ {\cal L}(E_1, \ldots, E_n;F) = {\cal L}(E_1, \ldots, E_n;F)$  for every $ F$ if and only if the completed projective tensor product $E_1\widehat{\otimes}_\pi\cdots\widehat{\otimes}_\pi E_n$ has the Schur property.
\end{center}
The Schur property on projective tensor products is a complete mystery. The following long standing problem (cf. \cite{bp, ggcambridge}) unfolds that we know nothing about it:

\medskip

\noindent {\bf Open Question 1.} Does $E\widehat{\otimes}_\pi F$ have the Schur property whenever $E$ and $F$ have the Schur property?

\medskip

So, we believe there is room for a closed hyper-ideal that {\it replicates} in the multilinear setting the essence of completely continuous linear operators and that is, at least formally, larger than ${\cal CC} \circ {\cal L}$. Here is our proposal:

\begin{definition}\rm Let $E_1, \ldots, E_n$ be Banach spaces. We say that a sequence $((x_j^{1},\ldots,x_j^{n}))_{j=1}^\infty$ in $E_1\times\cdots\times E_n$ {\it converges multilinearly} to $(x_1,\ldots,x_n) \in E_1\times\cdots\times E_n$ if  $T(x_j^{1},\ldots,x_j^{n})\longrightarrow T(x_1,\ldots,x_n)$ for every $n$-linear form $T\in \mathcal{L}(E_1,\ldots, E_n)$.\\
\indent We say that an operator $A\in \mathcal{L}(E_1,\ldots,E_n;F)$ is {\it multilinearly sequentially continuous} if $A(x_j^{1},\ldots,x_j^{n})\longrightarrow A(x_1,\ldots,x_n)$ in the norm of $F$ whenever $((x_j^{1},\ldots,x_j^{n}))_{j=1}^\infty$ converges multilinearly to $(x_1,\ldots,x_n)$. In this case we write $A\in \mathcal{L}_{msc}(E_1,\ldots,E_n;F)$.\end{definition}

\indent It is worth noting that multilinear convergence is no regular convergence. For example, let us see that the limit is not unique in general and that multilinear convergence does not imply coordinatewise boundedness:
\begin{ex}\rm Let $E$ be a Banach space and $0 \neq x \in E$. For every bilinear form $T \in {\cal L}(^2E)$,
$$T\left(jx, \frac{x}{j^2}\right) = \frac1j T(x,x) \stackrel{j}{\longrightarrow} 0 = T(0,y) = T(y,0),$$
for every $y \in E$. Thus, $(jx, \frac{x}{j^2})_{j=1}^\infty$ converges bilinearly to $(0,y)$ and to $(y,0)$ for every $y \in E$. Besides, the sequence $(jx)_{j=1}^\infty$ is unbounded in $E$.\end{ex}

Fortunately, for our purposes the following boundedness is sufficient:

\begin{lema}\label{ultimolema}Suppose that $((x_j^{1},\ldots,x_j^{n}))_{j=1}^\infty\subseteq E_1\times\cdots\times E_n$ converges multilinearly to $(x_1, \ldots, x_n)$. Then there exists $K>0$ such that $\|x_j^{1}\|\cdots\|x_j^{n}\|\leq K$ for every $j\in\mathbb{N}$ and $\|x_1\|\cdots\|x_n\|\leq K$.\end{lema}

\begin{proof}Let $\varphi\in (E_1\widehat{\otimes}_\pi\cdots\widehat{\otimes}_\pi E_n)'$. By the universal property of the projective tensor product (see, e.g., \cite[Theorem 2.9]{ryan}), there is an $n$-linear form $B\in\mathcal{L}(E_1, \ldots, E_n)$ such that $B(x_1, \ldots, x_n) = \varphi(x_1 \otimes \cdots \otimes x_n)$ for all $x_1 \in E_1, \ldots, x_n \in E_n$. From the multilinear convergence of $((x_j^{1},\ldots,x_j^{n}))_{j=1}^\infty$ we have \begin{eqnarray*}\varphi(x_j^{1}\otimes\cdots\otimes x_j^{n})=B(x_j^{1},\ldots,x_j^{n})\longrightarrow B(x_1,\ldots,x_n)=\varphi(x_1\otimes\cdots\otimes x_n).\end{eqnarray*} Thus $(x_j^{1}\otimes\cdots\otimes x_j^{n})_{j=1}^\infty$ is weakly convergent to  $x_1\otimes\cdots\otimes x_n$ in $E_1\widehat{\otimes}_\pi\cdots\widehat{\otimes}_\pi E_n$. By \cite[Proposition 3.5(iii)]{brezis} there exists $K>0$ such that $$\|x_j^{1}\|\cdots\|x_j^{n}\|=\pi(x_j^{1}\otimes\cdots\otimes x_j^{n})\le K$$ for every $j\in\mathbb{N}$, where $\pi(z)$ denotes the projective norm of $z \in E_1\widehat{\otimes}_\pi\cdots\widehat{\otimes}_\pi E_n$, and $$ \|x_1\| \cdots \|x_n\| = \pi(x_1 \otimes \cdots \otimes x_n) \leq \liminf_j  \|x_j^{1}\|\cdots\|x_j^{n}\| \leq K. $$\end{proof}

\begin{teo}\label{propmsc} The class $\mathcal{L}_{msc}$ of multilinearly sequentially continuous operators is a closed hyper-ideal, ${\cal CC} = {\cal L}_{msc}^1$ and $\mathcal{CC}\circ\mathcal{L}\subseteq\mathcal{L}_{msc}$.\end{teo}

\begin{proof} The equality ${\cal CC} = {\cal L}_{msc}^1$ follows immediately from the definitions. The proofs that $\mathcal{CC}\circ\mathcal{L}\subseteq\mathcal{L}_{msc}$ and that any component $\mathcal{L}_{msc}(E_1,\ldots,E_n;F)$ is a linear subspace of $\mathcal{L}(E_1,\ldots,E_n;F)$ are standard and we omit them. So, $\mathcal{L}_{msc}$ contains the finite rank operators. As to the hyper-ideal property, let $1\le m_1<\cdots<m_n$, $B_1\in \mathcal{L}(G_1,\ldots, G_{m_1};E_1), \ldots$,  $B_n\in \mathcal{L}(G_{m_{n-1}+1},\ldots, G_{m_n};E_n)$, $A \in \mathcal{L}_{msc}(E_1,\ldots, E_m;F)$ and $t \in \mathcal{L}(F;H)$ be given. Suppose that the sequence $((x_j^{1},\ldots,x_j^{m_n}))_{j=1}^\infty$ in $G_1\times\cdots\times G_{m_n}$ converges multilinearly to $(x_1,\ldots,x_{m_n})$.  For every $n$-linear form $T\in \mathcal{L}(E_1,\ldots,E_n)$, as  $T\circ(B_1,\ldots,B_n)\in\mathcal{L}(G_1,\ldots,G_{m_n})$, we have
\begin{align*}
T\circ(B_1,\ldots,B_n)(x_j^{1},\ldots,x_j^{m_n})&=
T\circ(B_1(x_j^{1},\ldots,x_j^{m_1}),\ldots,B_n(x_j^{m_{n-1}+1},\ldots, x_j^{m_n}))\\& \longrightarrow T\circ(B_1,\ldots,B_n)(x_1,\ldots,x_{m_n})\\&=T(B_1(x_1,\ldots,x_{m_1}),\ldots,B_n(x_{m_{n-1}+1},\ldots, x_{m_n})).
 \end{align*} This shows that the sequence $((B_1(x_j^{1},\ldots,x_j^{m_1}),\ldots,B_n(x_j^{m_{n-1}+1},\ldots, x_j^{m_n})))_{j=1}^\infty$ converges multilinearly to $(B_1(x_1,\ldots,x_{m_1}),\ldots,B_n(x_{m_{n-1}+1},\ldots, x_{m_n}))$. Since  $A$ is multilinear sequentially continuous, the sequence $$(A(B_1(x_j^{1},\ldots,x_j^{m_1}),\ldots,B_n(x_j^{m_{n-1}+1},\ldots, x_j^{m_n})))_{j=1}^\infty$$ converges in the norm of $F$ to $$A(B_1(x_1,\ldots,x_{m_1}),\ldots,B_n(x_{m_{n-1}+1},\ldots, x_{m_n})).$$ As $t$ is continuous, $(t\circ A(B_1(x_j^{1},\ldots,x_j^{m_1}),\ldots,B_n(x_j^{m_{n-1}+1},\ldots, x_j^{m_n})))_{j=1}^\infty$ converges in the norm of $H$ to $$t\circ A(B_1(x_1,\ldots,x_{m_1}),\ldots,B_n(x_{m_{n-1}+1},\ldots, x_{m_n})),$$ proving that $t\circ A\circ(B_1,\ldots,B_n)\in \mathcal{L}_{msc}(G_1,\ldots, G_{m_n};H)$.\\
\indent Now we prove that $\mathcal{L}_{msc}(E_1,\ldots,E_n;F)$ is closed in $\mathcal{L}(E_1,\ldots,E_n;F)$. To do so let $(A_j)_{j=1}^\infty \subseteq \mathcal{L}_{msc}(E_1,\ldots,E_n;F)$ be such that $A_j \stackrel{\|\cdot\|}{\longrightarrow}A\in\mathcal{L}(E_1,\ldots,E_n;F)$. 
Suppose that the sequence $((x_k^{1},\ldots,x_k^{n}))_{k=1}^\infty$ converges multilinearly to $(x_1,\ldots,x_n)\in E_1\times\cdots\times E_n$. From Lemma \ref{ultimolema} there is $K>0$ such that $\|x_k^{1}\|\cdots\|x_k^{n}\|\le K$ for every $k\in\mathbb{N}$ and $\|x_1\|\cdots\|x_n\|\le K$. Given $\varepsilon>0$, choose $j_0 \in \mathbb{N}$ such that $\|A-A_{j_0}\|< \frac{\varepsilon}{3K}$. As $A_{j_0}\in\mathcal{L}_{msc}(E_1,\ldots,E_n;F)$, there exists $k_0\in\mathbb{N}$ such that $$\|A_{j_0}(x_k^{1},\ldots,x_k^{n})-A_{j_0}(x_1,\ldots,x_n)\|<\dfrac{\varepsilon}{3}$$ for every $k\ge k_0$. Then,
\begin{align*}\|A(x_k^{1},&\ldots,x_k^{n})-A(x_1,\ldots,x_n)\|\\
&\le\|A(x_k^{1},\ldots,x_k^{n})-A_{j_0}(x_k^{1},\ldots,x_k^{n})\| +\|A_{j_0}(x_k^{1},\ldots,x_k^{n})-A_{j_0}(x_1,\ldots,x_n)\|\\
&~~~+\|A_{j_0}(x_1,\ldots,x_n)-A(x_1,\ldots,x_n)\|\\
&\le\|A-A_{j_0}\|\cdot\|x_k^{1}\|\cdots\|x_k^{n}\| +\|A_{j_0}(x_k^{1},\ldots,x_k^{n})-A_{j_0}(x_1,\ldots,x_n)\|\\&~~~+ \|A_{j_0}-A\|\cdot\|x_1\|\cdots\|x_n\|\\
&< \dfrac{\varepsilon}{3K}\cdot\|x_k^{1}\|\cdots\|x_k^{n}\|+\dfrac{\varepsilon}{3}+ \dfrac{\varepsilon}{3K}\cdot\|x_1\|\cdots\|x_n\|\le\varepsilon\end{align*} for every $k\ge k_0$. Thus $A(x_k^{1},\ldots,x_k^{n})\longrightarrow A(x_1,\ldots,x_n)$ in the norm of $F$, which proves that $A\in\mathcal{L}_{msc}(E_1,\ldots,E_n;F)$.\end{proof}

Sometimes the classes $\mathcal{CC}\circ\mathcal{L}$ and $\mathcal{L}_{msc}$ coincide:
\begin{prop}\label{propps}If $E_1\widehat{\otimes}_\pi\cdots\widehat{\otimes}_\pi E_n$ has the Schur property,
then $$\mathcal{L}_{msc}(E_1,\ldots,E_n;F)= \mathcal{CC}\circ\mathcal{L}(E_1,\ldots,E_n;F)=\mathcal{L}(E_1,\ldots,E_n;F),$$ for every Banach space $F$.\end{prop}

\begin{proof} As we have already mentioned, reasoning as in \cite[Proposition 3.3]{bp} we prove that ${\cal CC} \circ {\cal L}(E_1, \ldots, E_n;F) = {\cal L}(E_1, \ldots, E_n;F)$. From Theorem \ref{propmsc} we know that $\mathcal{CC}\circ\mathcal{L}\subset\mathcal{L}_{msc}$, what completes the proof.
\end{proof}

The only two cases we know that $E_1\widehat{\otimes}_\pi\cdots\widehat{\otimes}_\pi E_n$ has the Schur property are the following: (i) $\widehat{\otimes}_\pi^n \ell_1= \ell_1$ for every $n$; (ii) $BP \widehat{\otimes}_\pi BP$ has the Schur property, where $BP$ is the space constructed by Bourgain and Pisier \cite{bourgainpisier}. To the best of our knowledge, it is unknown whether $BP \widehat{\otimes}_\pi BP \widehat{\otimes}_\pi BP$ has the Schur property or not. Such few examples after so many years make us believe that the answer to Open Question 1 shall turn out to be negative. Although we cannot prove it, we believe that if $\mathcal{CC}\circ\mathcal{L}=\mathcal{L}_{msc}$, then the solution to Open Question 1 will be affirmative. That is why we pose the:

\medskip

\noindent{\bf Open Question 2.} We conjecture that $\mathcal{CC}\circ\mathcal{L}\subsetneqq\mathcal{L}_{msc}$.

\bigskip

\noindent{\bf Acknowledgement.} We thank Professor Mary Lilian Louren\c co for making our collaboration possible.

\vspace{2em}

\noindent Faculdade de Matem\'atica\\
Universidade Federal de Uberl\^andia\\
38.400-902 -- Uberl\^andia, Brazil\\
e-mails: botelho@ufu.br, ewerton@ime.usp.br.

\end{document}